\documentclass[a4paper]{amsart} 

\usepackage[all]{xy}

\usepackage[latin1]{inputenc}
\usepackage{amssymb}
\usepackage{amsmath, amsthm, amssymb}
\usepackage{hyperref}

\theoremstyle{plain}
\newtheorem{thm}{Theorem} 
\newtheorem{cor}[thm]{Corollary}
\newtheorem{prop}[thm]{Proposition}
\theoremstyle{definition}
\newtheorem{dfn}[thm]{Definition}
\newtheorem{ex}[thm]{Example}
\newtheorem{rem}[thm]{Remark}

 \newcommand{\NN}{\mathbb N}

 \def\gldim{\mathop{\rm gldim}}
 \def\pdim{\mathop{\rm pdim}}
 \def\Hom{\mathop{\rm Hom}}
 \def\rad{\mathop{\rm rad}}

\begin{document}

\title{Quivers without loops admit global dimension 2}


\author{Nicolas Poettering}
\address{Nicolas Poettering, Mathematisches Institut, Universität Bonn, Endenicher Allee 60, D-53115 Bonn, Germany}
\email{n.poettering@gmail.com}

\thanks{Supported by BIGS-Mathematics, Bonn and Mathematical Institute of the University Bonn}

\maketitle

\begin{abstract} Let $Q$ be a finite quiver without loops. Then there is an admissible ideal $I$ such that the algebra $kQ/I$ has global dimension at most two and is (strongly) quasi-hereditary. In addition some other (strongly) quasi-hereditary algebras $kQ/I'$ are constructed with bigger global dimension.
\end{abstract}

\section{Introduction}
 Let $Q$ be a finite quiver and $I$ an admissible ideal of the path algebra $kQ$. The global dimension $\gldim(kQ/I)$ of the finite-dimensional algebra $kQ/I$ is an important invariant of the category of finite-dimensional left $kQ/I$-modules. We are interested in the following question: Which natural numbers can actually occur as global dimensions of an algebra $kQ/I$ for some fixed finite quiver $Q$ with different admissible ideals $I$?

 It is well known that if $Q$ contains a loop, the global dimension of $kQ/I$ is infinite for any admissible ideal $I$. For a finite quiver $Q$ with at least one arrow an admissible ideal $I$ with $\gldim(kQ/I)=1$ exists if and only if $Q$ has no oriented cycles. In Corollary~\ref{cor-main} we answer this question for global dimension $2$. But we only consider small global dimensions since the global dimension of a finite-dimensional  algebra is neither bounded by a function depending on the number of simple modules (see~\cite{Green}) nor by a function on the Loewey length of the algebra (see~\cite{KK}).

\subsection{Main Result}
A quiver $R=(R_0,R_1,s_R,t_R)$ is a \textit{subquiver} of a quiver $Q=(Q_0,Q_1,s_Q,t_Q)$, if $R_0\subseteq Q_0$, $R_1\subseteq Q_1$, $s_R=s_Q|_{R_1}$ and $t_R=t_Q|_{R_1}$. For $m\in\NN_{>0}$ let $A_m$ be the quiver
\begin{align*}\vcenter{\begin{xy}\SelectTips{cm}{}
  \xymatrix@-0pc{
1\ar[r]&2\ar[r]&\ldots\ar[r]&m-1\ar[r]&m}
\end{xy}}\end{align*}
and $X_m$ the quiver
\begin{align*}\vcenter{\begin{xy}\SelectTips{cm}{}
  \xymatrix@-0pc{
1\ar[r]&2\ar[r]&\ldots\ar[r]&i\ar[r]&\ldots\ar[r]&m\ar@(ul,ur)[ll]}
\end{xy}}\end{align*} with some $1\leq i<m$.

\begin{samepage}
\begin{thm}\label{thm-main} Let $Q$ be a finite quiver without loops.
\begin{enumerate}
\item\label{part-n} Then there exists an admissible ideal $I$ such that $\gldim(kQ/I)\leq 2$.

\item\label{part-a} Let $k,m\in\NN$ with $2\leq k<m$. If $A_m$ is a subquiver of $Q$, then there exists an admissible ideal $I'$ such that $\gldim(kQ/I')=k$.

\item\label{part-x} Let $k,m\in\NN$ with $2\leq k\leq m$. If $X_m$ is a subquiver of $Q$, then there exists an admissible ideal $I'$ such that $\gldim(kQ/I')=k$.
\end{enumerate}
Moreover all these ideals are generated by zero relations of length at most three and the algebras are strongly quasi-hereditary.
\end{thm}
\end{samepage}

 The proof of this theorem is given in Section~\ref{sec-proof}. The following corollary is a direct consequence.

\begin{cor}\label{cor-main}  Let $Q$ be a finite quiver. Then there is an admissible ideal $I$ such that $\gldim(kQ/I)=2$ if and only if $Q$ has no loops and $(kQ^+)^2\neq0$.
\end{cor}

\begin{ex}
Let $Q$ be the quiver
$$\vcenter
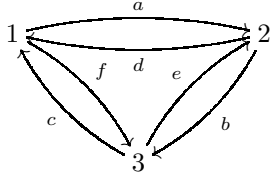
{\begin{xy}\SelectTips{cm}{}
  \xymatrix@+1pc{
1\ar@/^/[rr]^a\ar@/^/[dr]^f&&2\ar@/^/[dl]^b\ar@/^/[ll]^d\\
&3\ar@/^/[ul]^c\ar@/^/[ur]^e
  }
\end{xy}}$$
$I$ the admissible ideal $\langle da,cf,ef,cb,eb\rangle$ and $I'=I+\langle ba\rangle.$ Then $\gldim(kQ/I)=2$, $\gldim(kQ/I')=3$ and the algebras $kQ/I$ and $kQ/I'$ are strongly quasi-hereditary.
\end{ex}

\subsection{Acknowledgement}
I thank Dieter Happel for asking this nice question. He obtained most of these results independently. Additionally I thank Martin Kalck for discussing and reading my results.

\section{Definitions}
\subsection{Path algebra}
 Let $k$ be a field. Let $Q=(Q_0,Q_1,s,t)$ be a \textit{finite quiver}, i.e.\ a finite oriented graph with vertex set $Q_0$, arrow set $Q_1$ and maps $s,t\colon Q_1\to Q_0$ indicating the start and terminal point of each arrow. An \textit{oriented path} $\rho=a_1\ldots a_n$ in $Q$ is the concatenation of some arrows $a_1,\ldots,a_n\in Q_1$ such that $t(a_{i+1})=s(a_i)$ for all $1\leq i<n$. Additionally we introduce a path $e_i$ of length zero for each vertex $i\in Q_0$. The \textit{path algebra} $kQ$ of a quiver $Q$ is the $k$-vector space with the set of oriented paths as a basis. The product of basis vectors is given by the concatenation of paths if possible or by zero otherwise. Let $kQ^+$ be the ideal in the path algebra $kQ$, which is generated by all arrows in $Q_1$. An two sided ideal $I$ of the path algebra $kQ$ is called \textit{admissible} if there is a $k\in\NN_{>0}$ with $(kQ^+)^k\subseteq I\subseteq (kQ^+)^2$. 

 Let $Q$ be a finite quiver and $I$ an admissible ideal. Then $kQ/I$ is a finite-dimensional $k$-algebra. The isomorphism classes of simple left $kQ/I$-modules are in a unique bijection with the vertices $Q_0$ of the quiver $Q$. So we denote the simple module associated to $i\in Q_0$ by $S(i)$ and the projective cover of $S(i)$ by $P(i)$.

\subsection{Global dimension}
 Let $Q$ be a finite quiver, $I$ an admissible ideal and $M$ a $kQ/I$-module. The \textit{projective dimension} $\pdim(M)$ of the module $M$ is the length of a minimal projective resolution. The \textit{global dimension} $\gldim(kQ/I)$ of the algebra $kQ/I$ is the supremum of the projective dimensions of all modules. These dimensions are in general not finite. But it is well known that the global dimension of $kQ/I$ is the maximum of the projective dimensions of all simple modules.

\subsection{Quiver of a module}
 Let $Q$ be a finite quiver, $I$ an admissible ideal and $M$ a $kQ/I$-module. Take some Jordan Hölder filtration $0=M_0\subseteq M_1\subseteq M_2\subseteq\ldots\subseteq M_t=M$ of $M$ with $M_i/M_{i-1}\cong S(j(i))$ and $j(i)\in Q_0$ for all $1\leq i\leq t$. Thus for each $1\leq i\leq t$ a vector $v_i\in M_i-M_{i-1}$ exists such that $e_{j(i)}v_i=v_i$. Now we associate a quiver to the module $M$:

The vertices are the basis vectors $v_1,\ldots,v_t$. Let $v_i$ and $v_j$ be two vectors and $a\in Q_1$ such that $\lambda_j\neq0$ in $a v_i=\sum_k\lambda_kv_k$. In this case there is one arrow called $a$ from $v_i$ to $v_j$. Note that in general this quiver depends on the chosen basis.

For the quiver associated to $M$ we write $j(i)$ instead of $v_i$ for all $1\leq i\leq t$.

\subsection{Quasi-hereditary algebra} The algebras constructed in the proof of Theorem~\ref{thm-main} are strongly quasi-he\-red\-itary in the sense of Ringel~\cite{Ringel-sqh}.
\begin{dfn} Let $Q$ be a finite quiver with $Q_0=\{1,2,\ldots,n\}$ and $I$ an admissible ideal. We say that $kQ/I$ is \textit{strongly quasi-hereditary} if there is for any $i\in Q_0$ an exact sequence
\begin{align}\label{gl-sqh}0 \to R(i) \to P(i) \to\Delta(i)\to 0\end{align}
with the following two properties:
\begin{enumerate}
\item $R(i)$ is a direct sum of projective modules $P(j)$ with $j<i$.
\item If $S(j)$ is a composition factor of $\rad(\Delta(i))$, then $j>i$.
\end{enumerate}
\end{dfn}

In~\cite{Ringel-sqh} Ringel proved that a strongly quasi-hereditary algebra $kQ/I$ with $Q_0=\{1,2,\ldots,n\}$ is also quasi-hereditary and has global dimension at most $n$.


 \subsection{Examples}

\begin{ex}
 Let $Q$ be the quiver with one vertex $1$ and one loop $a$ and let $I=\langle a^3\rangle$. Then the quiver associated to  $S(1)$ and $P(1)$ are given by
$$1\qquad 1\stackrel a\to 1\stackrel a\to 1.$$
In this case $\pdim S(1)=\infty$, $\pdim P(1)=0$ and $\gldim(kQ/I)=\infty$. Thus the algebra $kQ/I$ is not strongly quasi-hereditary.
\end{ex}

\begin{ex}
Let $n\in\NN_{>0}$, $Q$ the quiver $1\stackrel{a_1}\longrightarrow2\stackrel{a_2}\longrightarrow\ldots\stackrel{a_{n-1}}\longrightarrow n$ and $I=\langle a_{i+1}a_i|1\leq i\leq n-2\rangle.$ Then the quivers of $S(i)$ and $P(j)$ with $i,j\in Q_0$ and $j\neq n$ are given by
$$i\qquad j\stackrel {a_j}\to j+1.$$
Then $\pdim S(i)=n-i$, $\gldim(kQ/I)=n-1$ and $kQ/I$ is strongly quasi-hereditary.
\end{ex}

\section{Proof of Theorem~\ref{thm-main}}\label{sec-proof}
The proof of the main theorem is divided into five propositions. The following relabeling of the vertices and arrows of $Q$ is possible for any finite quiver without loops.
Let $n,r_{ij}\in\NN$ and $Q$ be the quiver with \begin{align*}
Q_0&=\{1,\ldots,n\},\\
Q_1&=\{a_{ijt}\colon i\to j|i,j\in Q_0,t\in\NN,i\neq j,1\leq
t\leq r_{ij}\}.\end{align*} Let $I$ be the ideal
\begin{align}I=\langle a_{kit}a_{jku}|i,j,k\in Q_0,t,u\in\NN,i<k,j<k,1\leq t\leq r_{ki},1\leq u\leq r_{jk}\rangle.\label{eq-I} \end{align}
This ideal is admissible, since $(kQ^+)^{2n-1}\subseteq I$.

\begin{prop}\label{prop-n} Let $Q$ be a finite quiver without loops and $I$ the ideal defined in Equation~\eqref{eq-I}. Then $\gldim(kQ/I)\leq 2$.
\end{prop}
\begin{figure}[ht]
$$\vcenter{\begin{xy}\SelectTips{cm}{}
  \xymatrix@-1.3pc{
&&1\ar[dl]\ar[d]\ar[dr]\\
&2\ar[dl]\ar[d]&3\ar[d]&4\\
3\ar[d]&4&4\\
4}
\end{xy}}\quad
\vcenter{\begin{xy}\SelectTips{cm}{}
  \xymatrix@-1.3pc{
&2\ar[dl]\ar[d]\ar[dr]\\
P(1)&3\ar[d]&4\\
&4}
\end{xy}}\quad
\vcenter{\begin{xy}\SelectTips{cm}{}
  \xymatrix@-1.3pc{
&3\ar[dl]\ar[d]\ar[dr]\\
P(1)&P(2)&4}
\end{xy}}\quad
\vcenter{\begin{xy}\SelectTips{cm}{}
  \xymatrix@-1.3pc{
&4\ar[dl]\ar[d]\ar[dr]\\
P(1)&P(2)&P(3)}
\end{xy}}$$
\caption{The projective modules $P(i)$ with $i\in\{1,2,3,4\}$ for $n=4$ and $r_{ij}=1$.}
\label{fig-n4}
\end{figure}
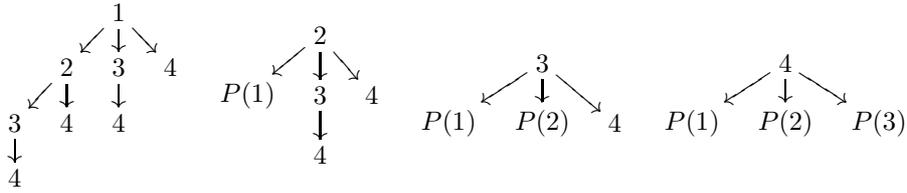

\begin{proof}Figure~\ref{fig-n4} shows the quivers of the indecomposable projective modules for $n=4$ and $r_{ij}=1$ for all $i,j\in Q_0$ with $i\neq j$. Using these pictures we get in general for each $i\in Q_0$ the following minimal projective resolution:
\begin{align}\label{eq-proj-res}
0\to \bigoplus_{j\in Q_0,j>i}\bigoplus_{k\in
Q_0,k<j}P(k)^{r_{ij}r_{jk}}\to \bigoplus_{j\in
Q_0-\{i\}}P(j)^{r_{ij}}\to P(i)\to&S(i)\to 0.\end{align} This means 
$\gldim(kQ/I)\leq2$.
\end{proof}

 Let $A_m$ be a subquiver of $Q$. Then $A_k$ with $k\leq m$ is also a subquiver of $Q$. If $A_m$ can be extended to some $X_m$ in $Q$ (i.e.\ there is a $a\in Q_1$ with $s(a)=m$ and $t(a)\in (A_m)_0$), then it is enough to study the case that $X_m$ is a subquiver of $Q$ (see Proposition~\ref{prop-x2}). Otherwise Part~\eqref{part-a} of Theorem~\ref{thm-main} is proven in the following proposition.

\begin{prop}\label{prop-a} Let $Q$ be a finite quiver without loops and $m\in\NN$ with $m\geq 3$. If $A_m$ is a subquiver of $Q$, which cannot be extended to some $X_m$ in $Q$, then there exists an admissible ideal $I'$ such that $\gldim(kQ/I')=m-1$.
\end{prop}

\begin{proof} Let $m\in\NN$ with $m\geq3$ and $A_m$ a subquiver of $Q$. Thus we have $r_{i,i+1}>0$ for all $1\leq i<m$. We use the ideal $I$ defined in~\eqref{eq-I} and Proposition~\ref{prop-n}. Let \begin{align*}I'=I+\langle a_{i+1,i+2,t}a_{i,i+1,u}|&i\in Q_0,t,u\in\NN,i\leq m-2,\\&1\leq t\leq r_{i+1,i+2},1\leq u\leq r_{i,i+1}\rangle.\end{align*} For $i\in Q_0$ define $R(i)=\bigoplus_{j\in Q_0,j<i}P(j)^{r_{ij}}$ and $\Delta(i)=P(i)/R(i)$ using that $R(i)$ is a direct summand of $\rad P(i)$. For $i\in Q_0$ with $i<m$ let $\Gamma(i)$ be the cokernel of some map $R(i)\oplus P(i+1)^{r_{i,i+1}}\to P(i)$ such that $\Gamma(i)$ is minimal with this property. Thus $\Gamma(i)$ is unique and the top of $\Gamma(i)$ is isomorphic to $S(i)$.

Again we give the quivers of indecomposable projective modules for $n=5$, $m=4$ and $r_{ij}=1$ for all $i,j\in Q_0$ with $i\neq j$ in Figure~\ref{fig-n5}. Figure~\ref{fig-n5-D} and \ref{fig-n5-G} shows the quivers of the  modules $\Delta(i)$ and the $\Gamma(i)$ in this case.
\begin{figure}[th]
$$\vcenter{\begin{xy}\SelectTips{cm}{}
  \xymatrix@-1.3pc{
&&&1\ar[dll]\ar[dl]\ar[dr]\ar[drr]\\
&2\ar[dl]\ar[d]&3\ar[d]\ar[dr]&&4\ar[d]&5\\
4\ar[d]&5&4\ar[d]&5&5\\
5&&5}
\end{xy}}\qquad
\vcenter{\begin{xy}\SelectTips{cm}{}
  \xymatrix@-1.3pc{
&2\ar[dl]\ar[d]\ar[dr]\ar[drr]\\
P(1)&3\ar[d]&4\ar[d]&5\\
&5&5}
\end{xy}}\qquad
\vcenter{\begin{xy}\SelectTips{cm}{}
  \xymatrix@-1.3pc{
&3\ar[dl]\ar[d]\ar[dr]\ar[drr]\\
P(1)&P(2)&4\ar[d]&5\\
&&5}
\end{xy}}$$
$$\vcenter{\begin{xy}\SelectTips{cm}{}
  \xymatrix@-1.3pc{
&4\ar[dl]\ar[d]\ar[dr]\ar[drr]\\
P(1)&P(2)&P(3)&5}
\end{xy}}\qquad
\vcenter{\begin{xy}\SelectTips{cm}{}
  \xymatrix@-1.3pc{
&5\ar[dl]\ar[d]\ar[dr]\ar[drr]\\
P(1)&P(2)&P(3)&P(4)}
\end{xy}}
$$
\caption{The projective modules $P(i)$ with $i\in\{1,2,3,4,5\}$ for $n=5$, $m=4$ and $r_{ij}=1$.}
\label{fig-n5}
\end{figure}
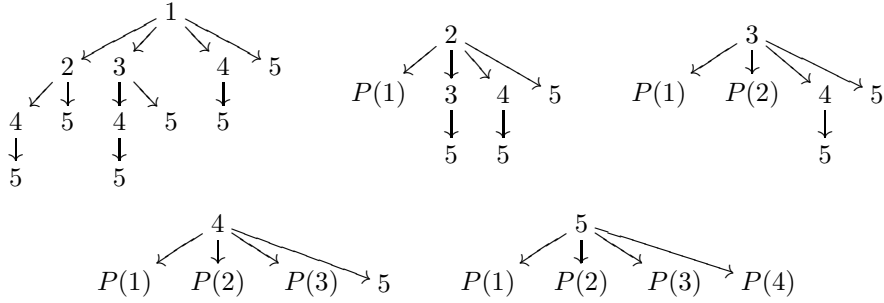
\begin{figure}[th]
$$P(1)\qquad
\vcenter{\begin{xy}\SelectTips{cm}{}
  \xymatrix@-1.3pc{
&2\ar[dl]\ar[d]\ar[dr]\\
3\ar[d]&4\ar[d]&5\\
5&5}
\end{xy}}\qquad
\vcenter{\begin{xy}\SelectTips{cm}{}
  \xymatrix@-1.3pc{
3\ar[d]\ar[dr]\\
4\ar[d]&5\\
5}
\end{xy}}\quad
\vcenter{\begin{xy}\SelectTips{cm}{}
  \xymatrix@-1.3pc{
4\ar[d]\\
5}
\end{xy}}\qquad
\vcenter{\begin{xy}\SelectTips{cm}{}
  \xymatrix@-1.3pc{
5}
\end{xy}}
$$
\caption{The modules $\Delta(i)$ with $i\in\{1,2,3,4,5\}$ for $n=5$, $m=4$ and $r_{ij}=1$.}
\label{fig-n5-D}
\end{figure}
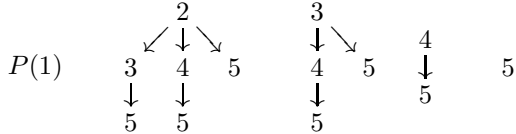
\begin{figure}[th]
$$ \vcenter{\begin{xy}\SelectTips{cm}{}
  \xymatrix@-1.3pc{
&&1\ar[dl]\ar[d]\ar[dr]\\
&3\ar[dl]\ar[d]&4\ar[d]&5\\
4\ar[d]&5&5\\
5}
\end{xy}}\qquad
\vcenter{\begin{xy}\SelectTips{cm}{}
  \xymatrix@-1.3pc{
2\ar[d]\ar[dr]\\
4\ar[d]&5\\
5}
\end{xy}}\qquad
\vcenter{\begin{xy}\SelectTips{cm}{}
  \xymatrix@-1.3pc{
3\ar[d]\\
5}
\end{xy}}
$$
 \caption{The modules $\Gamma(i)$ with $i\in\{1,2,3\}$ for $n=5$, $m=4$ and $r_{ij}=1$.}
 \label{fig-n5-G}
 \end{figure}
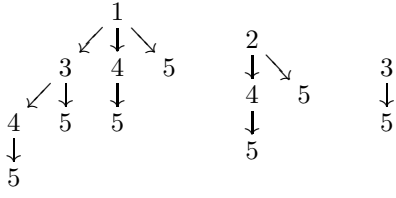

Since Sequence~\eqref{eq-proj-res} is exact for all $i\in Q_0$ with $i\geq m-1$ we get $\pdim(S(i))\in\{0,1,2\}$ in this case. With Figure~\ref{fig-n5}, \ref{fig-n5-D} and~\ref{fig-n5-G} we get in general for each $i\in Q_0$ with $i\leq m-2$ an exact sequence
 \begin{align}0\to R(i)\oplus\Gamma(i+1)^{r_{i,i+1}}\oplus\bigoplus_{j\in Q_0,j\geq i+2}\Delta(j)^{r_{ij}}\to P(i)\to S(i)\to0.\label{eq-si-aufl}
 \end{align}
 Using $\pdim(\Gamma(i+1))\geq 1$ for all $i\in Q_0$ with $i\leq m-2$ we get $\pdim(S(i))=\pdim(\Gamma(i+1))+1$.  Since in this case the sequence
\begin{align}\label{eq-gi-aufl}
0\to R(i)\oplus\Gamma(i+1)^{r_{i,i+1}}\to P(i)\to \Gamma(i)\to0\end{align} is exact we know 
$\pdim(\Gamma(i))=\pdim(\Gamma(i+1))+1$. Thus for $i\in Q_0$ with $i\leq m-2$
\begin{align*}
\pdim(S(i))=\pdim(\Gamma(i+1))+1=\pdim(\Gamma(m-1))+m-1-i.\end{align*}
 Therefore $\pdim(S(i))\leq\pdim(S(1))$ for all $i\in Q_0$ and $$\gldim(kQ/I')=\pdim(\Gamma(m-1))+m-2.$$
It is not hard to compute the minimal projective resolution of $\Gamma(m-1)$:
\begin{align}
\label{eq-gamma-m-1}0\to R(m)^{r_{m-1,m}}\to R(m-1)\oplus P(m)^{r_{m-1,m}}\to P(m-1)\to \Gamma(m-1)\to0.
\end{align}

 Now we assume that $A_m$ cannot be extended to some $X_m$ in $Q$. Thus $R(m)=0$. Then Sequence~\eqref{eq-gamma-m-1} provides $\pdim(\Gamma(m-1))=1$ and $\gldim(kQ/I')=m-1$.
\end{proof}

\begin{rem}
The proof of Proposition~\ref{prop-a} holds up to Sequence~\eqref{eq-gamma-m-1} for general subquivers $A_m$. Hence we can use it for $X_m$, too.
\end{rem}

Let $X_m$ be a subquiver of $Q$ and $i\in(X_m)_0$ with $r_{mi}\neq0$. So we can assume that $i$ is chosen maximal for fixed vertices $(X_m)_0=\{1,\ldots,m\}$. Thus (by relabeling the vertices) $X_{m-j}$ is also a subquiver of $Q$ for all $0\leq j<i$ and $A_{m-i}$ as well. Using Proposition~\ref{prop-a} the following two propositions yield Part~\eqref{part-x} of Theorem~\ref{thm-main}.

\begin{prop}\label{prop-x1} Let $Q$ be a finite quiver without loops and $m\in\NN$ with $m\geq 2$. If $X_m$ is a subquiver of $Q$, then there exists an admissible ideal $I'$ such that $\gldim(kQ/I')=m$.
\end{prop}

\begin{proof} Using the proof of Proposition~\ref{prop-a}, especially the notation, $R(m)\neq0$. Then in this case Sequence~\eqref{eq-gamma-m-1} yields $\pdim(\Gamma(m-1))=2$ and $\gldim(kQ/I')=m$.
\end{proof}

\begin{prop}\label{prop-x2}Let $Q$ be a finite quiver without loops and $m\in\NN$ with $m\geq 4$. If $X_m$ is a subquiver of $Q$, then there exists an admissible ideal $I''$ such that $\gldim(kQ/I'')=m-1$.
\end{prop}

\begin{proof}
 This proof is done very similar to the proof of Proposition~\ref{prop-a}. Let \begin{align*}
I''=I+\langle a_{i+1,i+2,t}a_{i,i+1,u}|
&i\in Q_0,t,u\in\NN,i\leq m-4,\\
&1\leq t\leq r_{i+1,i+2},1\leq u\leq r_{i,i+1}\rangle\\
+\langle a_{m-1,m,s}a_{m-2,m-1,t}a_{m-3,m-2,u}|
&s,t,u\in\NN,1\leq s\leq r_{m-1,m},\\
&1\leq t\leq r_{m-2,m-1},1\leq u\leq r_{m-3,m-2}\rangle.
\end{align*} 
We define $R(i)$, $\Delta(i)$ and $\Gamma(i)$ as in the proof of Proposition~\ref{prop-a}. Thus $r_{i,i+1}>0$ for all $1\leq i<m$ and $R(m)\neq0$. Again we give some quivers of indecomposable projective modules for $n=5$, $m=4$ and $r_{ij}=1$ for all $i,j\in Q_0$ with $i\neq j$ in Figure~\ref{fig-n5-2}. 
\begin{figure}[th]
$$\vcenter{\begin{xy}\SelectTips{cm}{}
  \xymatrix@-1.3pc{
&&&1\ar[dll]\ar[d]\ar[drr]\ar[drrr]\\
&2\ar[dl]\ar[d]\ar[dr]&&3\ar[d]\ar[dr]&&4\ar[d]&5\\
3\ar[d]&4\ar[d]&5&4\ar[d]&5&5\\
5&5&&5}
\end{xy}}\qquad
\vcenter{\begin{xy}\SelectTips{cm}{}
  \xymatrix@-1.3pc{
&2\ar[dl]\ar[d]\ar[dr]\ar[drr]\\
P(1)&3\ar[d]\ar[dl]&4\ar[d]&5\\
4\ar[d]&5&5\\5}
\end{xy}}$$
\caption{The projective modules $P(1)$ and $P(2)$ for $n=5$, $m=4$ and $r_{ij}=1$.}
\label{fig-n5-2}
\end{figure}
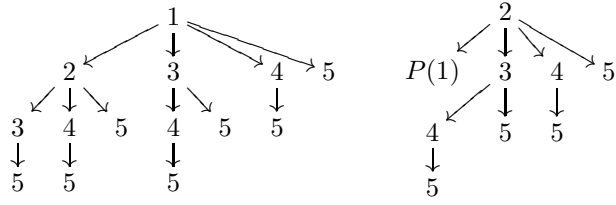

Since Sequence~\eqref{eq-proj-res} is exact for all $i\in Q_0$ with $i\geq m-2$ we get $\pdim(S(i))\in\{0,1,2\}$ in this case. For $i=m-3$ we construct the following exact sequences with $r={r_{m-3,m-2}r_{m-2,m-1}r_{m-1,m}}$:
\begin{align*}
     0
 &\to R(m)^r
 \to R(m-2)^{r_{m-3,m-2}}\oplus P(m)^r\\
 &\to R(m-3)\oplus P(m-2)^{r_{m-3,m-2}}\oplus\bigoplus_{j\in Q_0,j\geq m-1}\Delta(j)^{r_{m-3,j}}\\
 &\to P(m-3)
 \to S(m-3)
 \to 0,\\
     0
 &\to R(m)^r
 \to R(m-2)^{r_{m-3,m-2}}\oplus P(m)^r\\
 &\to R(m-3)\oplus P(m-2)^{r_{m-3,m-2}}
\to P(m-3)\to \Gamma(m-3)\to0.
\end{align*}
 Since $\pdim(\Delta(j))\leq1$ and $R(m)\neq 0$ we get $\pdim(S(m-3))=\pdim(\Gamma(m-3))=3$.  With Figure~\ref{fig-n5-2} Sequence~\eqref{eq-si-aufl} and~\eqref{eq-gi-aufl} are exact again for each $i\in Q_0$ with $i\leq m-4$. Thus in this case $\pdim(\Gamma(i))=\pdim(\Gamma(i+1))+1$ and
\begin{align*}
\pdim(S(i))=\pdim(\Gamma(i+1))+1=\pdim(\Gamma(m-3))+m-3-i=m-i.\end{align*}
Therefore $\gldim(kQ/I'')=m-1.$
\end{proof}
The following proposition completes the proof of Theorem~\ref{thm-main}.

\begin{prop} The algebras occurring in Proposition~\ref{prop-n}, \ref{prop-a}, \ref{prop-x1} and \ref{prop-x2} are strongly quasi-hereditary.
\end{prop}

\begin{proof} Define $R(i)$ and $\Delta(i)$ as above. Thus Sequence~\eqref{gl-sqh} is exact and by the construction of the ideal $I$ holds $\dim_k{\Hom}_{kQ/I}(P(j),\Delta(i))=\delta_{ij}$ for all $j\in Q_0$ with $j\leq i$.
\end{proof}

\end{document}